\documentclass[dvips,a4paper,twoside,10pt]{article}
\usepackage{amsmath,amsthm,amssymb,mathrsfs}
\usepackage{bbm}
\usepackage{graphicx,psfrag}
\usepackage[ps2pdf=true,colorlinks=true,linkcolor=blue,citecolor=red,bookmarks=true,bookmarksnumbered=true]{hyperref}
\usepackage[numbers,comma]{natbib}
\usepackage{fancyhdr}
\usepackage{titlesec}
\usepackage[a4paper,centering,bindingoffset=1cm,marginpar=2cm]{geometry}
\usepackage[font=normalsize,format=plain,labelfont=sc,textfont=up,width=0.75\textwidth,labelsep=period]{caption}
\titleformat{\section}{\filcenter\sc\large}{\thesection.\;}{0em}{}
\titleformat{\subsection}[runin]{\bf}{\thesubsection.\;}{0em}{}[.]

\hypersetup
{
    pdfauthor={Peter Elbau},
    pdfsubject={},
    pdftitle={Sequential Lower Semi-Continuity of Non-Local Functionals},
    pdfkeywords={Variational Calculus, Non-Local Functionals, Lower Semi-Continuity}
}

\newcommand{\R}{\mathbbm R}
\newcommand{\Q}{\mathbbm Q}
\newcommand{\Z}{\mathbbm Z}
\newcommand{\Rinf}{\mathbbm R\cup\{\infty\}}
\newcommand{\N}{\mathbbm N}
\renewcommand{\L}[1]{\mathscr L^{#1}}

\renewcommand{\d}{\,\mathrm{d}}
\renewcommand{\c}{^\mathrm{c}}
\newcommand{\J}{\mathcal J^p_f}

\newcommand{\p}{\,\mathrm p^p_M}
\newcommand{\pp}[2]{\,\mathrm p^{#1}_{#2}}
\renewcommand{\mid}{\,:\,}
\let\orgtimes\times
\renewcommand{\times}{\!\orgtimes\!}
\let\orgsetminus\setminus
\renewcommand{\setminus}{\!\orgsetminus\!}

\newcommand{\lsC}{sequential lower semi-continuity}

\newtheorem{theorem}{Theorem}
\newtheorem{lemma}[theorem]{Lemma}
\newtheorem{proposition}[theorem]{Proposition}
\newtheorem{corollary}[theorem]{Corollary}

\theoremstyle{definition}
\newtheorem{definition}[theorem]{Definition}

\theoremstyle{remark}

\newtheorem{example}[theorem]{Example}

\title{Sequential Lower Semi-Continuity of Non-Local Functionals}
\author{Peter Elbau\thanks{Johann Radon Institute for Computational and Applied Mathematics (RICAM), Linz, Austria, e-mail: peter.elbau@ricam.oeaw.ac.at}}
\date{}

\begin{document}
\maketitle
\thispagestyle{empty}

\begin{abstract}
We give a characterisation for non-local functionals 
\[ \mathcal J:L^p(X;\R^n)\to\R\cup\{\infty\},\quad\mathcal J(u) = \int_X\int_Xf(x,y,u(x),u(y))\d x\d y \]
on Lebesgue spaces to be weakly sequentially lower semi-continuous. Essentially, the requirement is that the functions
\[ \Phi_{x,\psi}:\R^n\to\R,\quad \Phi_{x,\psi}(w)=\int_Xf(x,y,w,\psi(y))\d y \]
are for every $\psi\in L^p(X;\R^n)$ for almost all $x\in X$ convex.

Moreover, we show that this condition is in the case $n=1$ (up to some equivalence in the integrand $f$) equivalent to the separate convexity of the function $(w,z)\mapsto f(x,y,w,z)$ for almost all $(x,y)\in X\times X$. 
\end{abstract}

\section{Introduction}
The purpose of these notes is to study the properties of non-local functionals of the form
\begin{equation}\label{eqNonLocalFunctional}
\mathcal J:L^p(X;\R^n)\to\R\cup\{\infty\},\quad\mathcal J(u)=\int_X\int_Xf(x,y,u(x),u(y))\d x\d y,
\end{equation}
especially regarding the existence of minimising points. 

Such kind of functionals recently appeared in a derivative-free characterisation of the Sobolev and the total variation seminorm~\cite{BourgainBrezisMironescu2002,Ponce2004}. More precisely, it was shown that these seminorms can be written as the limit of a sequence of non-local functionals which essentially emerge from replacing the derivative in the seminorm by a difference quotient. As an application of this result, it became possible to reformulate variational problems such as e.g.\ the total variation regularisation for image denoising~\cite{RudinOsherFatemi1992} by approximating the seminorm therein with the corresponding non-local functional, thus leading to variational problems for non-local functionals~\cite{AubertKornprobst2009,PontowScherzer2009}.

As another example where such non-local variational problems arose, we mention the variational formulation of neighbourhood filters~\cite{KindermannOsherJones2005}. In this case, the non-locality of the functional was utilised to measure and, by minimising the functional, also to enforce similarities of different regions in an image. For an overview of non-local functionals recently introduced in image analysis, we refer to~\cite{BoulangerElbauPontow2010}.

The main interest of this paper is the existence of minimising points of such functionals. Following the direct method in the calculus of variations, the existence can be guaranteed by imposing the condition that $\mathcal J$ is coercive, meaning that $\mathcal J(u)\to\infty$ whenever $\|u\|_p\to\infty$, and that $\mathcal J$ is sequentially lower semi-continuous with respect to the weak topology on $L^p(X;\R^n)$ if $p\in(1,\infty)$ and with respect to the weak-star topology on $L^\infty(X;\R^n)$ if $p=\infty$. Our aim is therefore to find a good characterisation for the sequential lower semi-continuity of a non-local functional.

Before doing so, we take in Section~\ref{seIntegrability} a closer look at the conditions which we need to ensure that the function $(x,y)\mapsto f(x,y,u(x),u(y))$ is integrable for all functions $u\in L^p(X;\R^n)$. It turns out that we can reach integrability with a weaker condition than the natural estimate of the form
\[ |f(x,y,w,z)|\le \alpha(x,y)+\beta(x)|z|^p+\beta(y)|w|^p+C|w|^p|z|^p \]
with $\alpha\in L^1(X\times X)$, $\beta\in L^1(X)$, $C\in\R$, unlike in the case of local functionals (i.e.\ functionals of the form $\mathcal J_{\mathrm{local}}:L^p(X;\R^n)\to\R$, $\mathcal J_{\mathrm{local}}(u)=\int_Xf_{\mathrm{local}}(x,u(x))\d x$) where this sort of bound is equivalent to the integrability, see e.g.\ Theorem~$6.45$ in~\cite{FonsecaLeoni2007}.

Then we turn to the sequential lower semi-continuity of non-local functionals. We show in Section~\ref{seStrongSemiContinuity} that the functional $\mathcal J$ is sequentially lower semi-continuous with respect to the strong topology if (in addition to some lower bound for the function $f$) the map $(w,z)\mapsto f(x,y,w,z)$ is for almost all $(x,y)\in X\times X$ lower semi-continuous. Afterwards, we establish in Section~\ref{seWeakSemiContinuity} the result that a non-local functional, fulfilling some regularity assumptions, is sequentially lower semi-continuous with respect to the weak topology on $L^p(X;\R^n)$ if $p\in(1,\infty)$ and with respect to the weak-star topology on $L^\infty(X;\R^n)$ if $p=\infty$ if and only if the function
\[ \Phi_{x,\psi}:\R^n\to\R,\quad \Phi_{x,\psi}(w)=\int_Xf(x,y,w,\psi(y))\d y \]
is for every $\psi\in L^p(X;\R^n)$ for almost all $x\in X$ convex.

Finally, we discuss in Section~\ref{seEquivalentIntegrands} which functions $f$ lead to the same non-local functional. We make use of this ambiguity in the integrand to show that in the case $n=1$ the condition that $\Phi_{x,\psi}$ is for every $\psi\in L^p(X)$ for almost all $x\in X$ convex is (for sufficiently regular functions $f$) equivalent to the fact that there exists a function $\tilde f$, defining the same non-local functional as $f$, such that the map $(w,z)\mapsto\tilde f(x,y,w,z)$ is for almost all $(x,y)\in X\times X$ separately convex.

The first results in this direction (formulated on Sobolev instead of Lebesgue spaces) are going back to Pablo Pedregal~\cite{Pedregal1997} where he gave an equivalent characterisation of the sequential lower semi-continuity in terms of Jensen type inequalities for the integrand. In later papers~\cite{Munoz2000,BevanPedregal2006,Munoz2009}, it was further shown that at least in the homogeneous case, i.e.\ if the integrand is not explicitly depending on the variables $x$ and $y$, the functional is sequentially lower semi-continuous if and only if the integrand is separately convex.

\section{Definition of Non-Local Functionals}
Throughout the paper, let $X$ be a bounded, Lebesgue measurable subset of $\R^m$, $m,n\in\N$, and $p\in[1,\infty]$. We consider $X$ as the measure space defined by the Lebesgue measure $\L m$ on the $\sigma$-algebra of all Lebesgue measurable subsets of $X$. Moreover, we consider $\R^N$ for every $N\in\N$ as the measure space defined by the Lebesgue measure $\L N$ on the Borel $\sigma$-algebra of $\R^N$.

Let us now clarify what we mean by a non-local functional. 
Since we can interchange the order of integration in~\eqref{eqNonLocalFunctional}, we may restrict our attention to functions $f$ which are pairwise symmetric.
\begin{definition}
We call a function $f:X\times X\times\R^n\times\R^n\to\R$ pairwise symmetric if
\begin{equation}\label{eqSymmetry}
f(x,y,w,z) = f(y,x,z,w)\quad\text{for all}\quad x,y\in X,\quad w,z\in\R^n.
\end{equation}
\end{definition}

\begin{definition}\label{deNonLocalFunctional}
Let $f:X\times X\times\R^n\times\R^n\to\R$ be a pairwise symmetric, measurable function (with respect to the product $\sigma$-algebra of the chosen $\sigma$-algebras on $X$ and $\R^n$)
whose negative part $f^-$ fulfils that
\begin{equation}\label{eqIntegrabilityNegativePart}
\int_X\int_Xf^-(x,y,u(x),u(y))\d x\d y < \infty\quad\text{for all}\quad u\in L^p(X;\R^n).
\end{equation}
Then we call
\[ \J:L^p(X;\R^n)\to\Rinf,\quad\J(u)=\int_X\int_Xf(x,y,u(x),u(y))\d x\d y \]
the non-local functional on $L^p(X;\R^n)$ defined by the function $f$.
\end{definition}

The pairwise symmetry of the integrand $f$ is just introduced for convenience since a function $\tilde f$ and its symmetrisation $f$, 
\[ f(x,y,w,z)=\frac12(\tilde f(x,y,w,z)+\tilde f(y,x,z,w)), \]
would anyway define the same non-local functional.

We remark that the measurability of the function $f$ guarantees that also the composition $f\circ g_{\varphi,\psi}$ of $f$ and the measurable function
\[ g_{\varphi,\psi}:X\times X\to X\times X\times\R^n\times\R^n,\quad g_{\varphi,\psi}(x,y)=(x,y,\varphi(x),\psi(y)) \]
is for all measurable functions $\varphi,\psi:X\to\R^n$ again measurable, so that the integrals in Definition~\ref{deNonLocalFunctional} are well-defined.

In the following, we will try to characterise the functions $f$ which fulfil the condition~\eqref{eqIntegrabilityNegativePart}. To formulate the results for the values $p\in[1,\infty)$ and for $p=\infty$ in the same way, let us introduce the function $\pp{q}{M}:\R^n\to[0,\infty]$, $q\in[0,\infty]$, $M\in(0,\infty)$, by
\begin{equation}\label{eqFunctionP}
\pp{q}{M}(w) = |w|^q\quad\text{if}\;q\in[0,\infty)\quad\text{and}\quad\pp{\infty}{M}(w)=\begin{cases}\infty&\text{if}\;|w|>M,\\0&\text{if}\;|w|\le M,\end{cases}
\end{equation}
which has the nice property that
\begin{equation}\label{eqPropertyOfP}
\int_X\p(\varphi(x))\d x = \p(\|\varphi\|_p)
\end{equation}
for every $M\in(0,\infty)$ and $\varphi\in L^p(X;\R^n)$.

\section{Integrability Conditions}\label{seIntegrability}
We are looking for a criterion for a measurable function $f:X\times X\times\R^n\times\R^n\to[0,\infty)$ to fulfil the condition
\begin{equation}\label{eqIntegrability}
\int_X\int_Xf(x,y,u(x),u(y))\d x\d y < \infty\quad\text{for every}\quad u\in L^p(X;\R^n).
\end{equation}
Since the fact that we use in the two last components of the integrand $f$ the same function~$u$ is only relevant when the values $x$ and $y$ are close to each other, we may try to consider instead of~\eqref{eqIntegrability} the stronger condition where we impose integrability also for different functions~$\varphi,\psi\in L^p(X;\R^n)$ in these two components.

\begin{proposition}\label{thIntegrability}
Let $f:X\times X\times\R^n\times\R^n\to[0,\infty)$ be a measurable function.

Then the following statements are equivalent:
\begin{enumerate}
\item
The function $f$ fulfils
\[ \int_X\int_Xf(x,y,u(x),u(y))\d x\d y < \infty\quad\text{for all}\quad u\in L^p(X;\R^n). \]
\item
The function $f$ fulfils
\[ \int_X\int_Xf(x,y,\varphi(x),\psi(y))\d x\d y<\infty\quad\text{for all}\quad\varphi,\psi\in L^p(X;\R^n). \]
\item
There exists for every $\psi\in L^p(X;\R^n)$ and every $M\in(0,\infty)$ a function $\alpha_{M,\psi}\in L^1(X)$ and a constant $C_\psi\in(0,\infty)$ such that
\begin{equation}\label{eqIntegrabilityCondition}
\int_Xf(x,y,w,\psi(y))\d y\le\alpha_{M,\psi}(x)+C_\psi\p(w)
\end{equation}
for almost all $x\in X$ and all $w\in\R^n$.
\end{enumerate}
\end{proposition}

\begin{proof}
We start with the implication from (i) to (ii). 
Let us assume that we find two functions $\varphi,\psi\in L^p(X;\R^n)$ such that the measurable function $g:X\times X\to[0,\infty)$, $g(x,y)=f(x,y,\varphi(x),\psi(y))$ fulfils
\[ \int_X\int_Xg(x,y)\d x\d y = \infty. \]
To prove the implication, it is enough to construct a function $u\in L^p(X;\R^n)$ with
\begin{equation}\label{eqDivergencyAtU}
\int_X\int_Xf(x,y,u(x),u(y))\d x\d y = \infty.
\end{equation}

If there exists a measurable set $A\subset X$ with $\int_A\int_{A\c}g(x,y)\d y\d x=\infty$, $A\c=X\setminus A$, we can simply choose $u=\chi_A\varphi+\chi_{A\c}\psi$, to get~\eqref{eqDivergencyAtU}. 
Otherwise, if no such set $A$ exists, we find a decreasing sequence $(A_k)_{k\in\N}$ of measurable sets $A_k\subset X$, $k\in\N$, such that $A_{k+1}\subset A_k$, $\L m(A_{k+1})=\frac12\L m(A_k)$, and 
\[ \int_{A_k}\int_{A_k}g(x,y)\d x\d y = \infty\quad\text{for all}\quad k\in\N. \]
We next choose a subsequence $(A_{k_\ell})_{\ell\in\N}$ of $(A_k)_{k\in\N}$ such that the pairwise disjoint sets $\tilde A_\ell=A_{k_\ell}\setminus A_{k_{\ell+1}}$, $\ell\in\N$, fulfil
\[ \int_{\tilde A_\ell}\int_{\tilde A_\ell}g(x,y)\d x\d y\ge\frac1\ell\quad\text{for all}\quad\ell\in\N. \]
We further divide for every $\ell\in\N$ the set $\tilde A_\ell\times\tilde A_\ell$ into the measurable sets
\begin{equation}\label{eqDefinitionOfEjl}
E_{j,\ell}=\{(x,y)\in\tilde A_\ell\times\tilde A_\ell\mid g(x,y)\in[j-1,j)\},\quad j\in\N.
\end{equation}
Then we find for every $\ell\in\N$ a constant $N_\ell\in\N$ such that
\begin{equation}\label{eqChoiceOfNl}
\sum_{j=1}^{N_\ell}j\L m(E_{j,\ell}) \ge \frac1{2\ell}.
\end{equation}

Using now Lemma~\ref{thCheckerboardApproximation}, we get for every $\ell\in\N$ a $\delta_\ell\in(0,\infty)$ such that the checkerboard pattern $S_{\delta_\ell}\subset\R^m$ defined in~\eqref{eqCheckerboardTiling} fulfils
\begin{equation}\label{eqCheckerboardEstimate}
\L{2m}((S_{\delta_\ell}\times S_{\delta_\ell}\c)\cap E_{j,\ell})\ge\frac18\L{2m}(E_{j,\ell})
\end{equation}
for all $j\in[1,N_\ell]\cap\N$.

Setting finally $S = \bigsqcup_{\ell\in\N}(S_{\delta_\ell}\cap\tilde A_\ell)\subset X$ and $u = \chi_S\varphi+\chi_{S\c}\psi$, where $S\c=X\setminus S$, we get~\eqref{eqDivergencyAtU}. Indeed, since $S\c\supset\bigsqcup_{\ell\in\N}(S_{\delta_\ell}\c\cap\tilde A_\ell)$, we have
\[ \int_X\int_X f(x,y,u(x),u(y))\d x\d y \ge\int_S\int_{S\c}g(x,y)\d y\d x\ge\sum_{\ell=1}^\infty\int_{S_{\delta_\ell}\cap\tilde A_\ell}\int_{S_{\delta_\ell}\c\cap\tilde A_\ell}g(x,y)\d y\d x. \]
Using then $\tilde A_\ell\times\tilde A_\ell\supset\bigsqcup_{j=1}^{N_\ell}E_{j,\ell}$, we get by definition~\eqref{eqDefinitionOfEjl} of the sets~$E_{j,\ell}$ that
\begin{align*}
\sum_{\ell=1}^\infty\int_{S_{\delta_\ell}\cap\tilde A_\ell}\int_{S_{\delta_\ell}\c\cap\tilde A_\ell}g(x,y)\d y\d x
&\ge\sum_{\ell=1}^\infty\sum_{j=1}^{N_\ell}\iint_{(S_{\delta_\ell}\times S_{\delta_\ell}\c)\cap E_{j,\ell}}g(x,y)\d x\d y \\
&\ge\sum_{\ell=1}^\infty\sum_{j=1}^{N_\ell}(j-1)\L{2m}((S_{\delta_\ell}\times S_{\delta_\ell}\c)\cap E_{j,\ell}).
\end{align*}
Taking now the estimate~\eqref{eqCheckerboardEstimate} into account, we find with our choice~\eqref{eqChoiceOfNl} of $N_\ell$ that
\begin{align*}
\sum_{\ell=1}^\infty\sum_{j=1}^{N_\ell}(j-1)\L{2m}((S_{\delta_\ell}\times S_{\delta_\ell}\c)\cap E_{j,\ell})
&\ge\frac18\sum_{\ell=1}^\infty\sum_{j=1}^{N_\ell}(j-1)\L{2m}(E_{j,\ell}) \\
&\ge\frac18\left(\sum_{\ell=1}^\infty\frac1{2\ell}-\L m(X)^2\right) = \infty.
\end{align*}
This concludes the proof that (i) implies~(ii). The converse direction is trivial.

\medskip

Given condition (ii), we can use the result for local functionals. Indeed, we know that
\[ \int_XF_\psi(x,\varphi(x))\d x<\infty  \]
for all $\varphi\in L^p(X;\R^n)$, where the function $F_\psi$ is for all $\psi\in L^p(X;\R^n)$ defined by
\[ F_\psi:X\times\R^n\to\Rinf,\quad F_\psi(x,w)=\int_Xf(x,y,w,\psi(y))\d y. \]
Therefore, in the case $p\in[1,\infty)$, there exist for every $\psi\in L^p(X;\R^n)$ a function $\alpha_\psi\in L^1(X)$ and a constant $C_\psi\in(0,\infty)$ such that
\begin{equation}\label{eqIntegrabilityCondition1}
F_\psi(x,w)\le\alpha_\psi(x)+C_\psi|w|^p
\end{equation}
for almost all $x\in X$ and all $w\in\R^n$, see e.g.\ Theorem~$6.45$ in~\cite{FonsecaLeoni2007}.
If $p=\infty$, we find for every $\psi\in L^\infty(X;\R^n)$ and every constant $M\in(0,\infty)$ a function $\alpha_{M,\psi}\in L^1(X)$ such that
\begin{equation}\label{eqIntegrabilityCondition2}
F_\psi(x,w)\le\alpha_{M,\psi}(x)
\end{equation}
for almost all $x\in X$ and all $w\in\R^n$ with $|w|\le M$, see e.g.\ Theorem~$6.47$ in~\cite{FonsecaLeoni2007}.
Using the function $\p$ defined in~\eqref{eqFunctionP}, the conditions~\eqref{eqIntegrabilityCondition1} and~\eqref{eqIntegrabilityCondition2} can be written in the condensed form~\eqref{eqIntegrabilityCondition}.

\medskip

The implication from (iii) to (ii) is finally found by direct calculation. From condition~\eqref{eqIntegrabilityCondition}, we get with the property~\eqref{eqPropertyOfP} of the function $\p$ that 
\[ \int_X\int_Xf(x,y,\varphi(x),\psi(y))\d y\d x\le \|\alpha_{M,\psi}\|_1+C_\psi\p(\|\varphi\|_p)<\infty \]
for all $\varphi,\psi\in L^p(X;\R^n)$ if we choose in the case $p=\infty$ the constant $M\ge\|\varphi\|_\infty$.
\end{proof}

We remark that condition~\eqref{eqIntegrabilityCondition} allows for non-integrable divergencies in the function $(x,y)\mapsto\sup_{w,z\in B}f(x,y,w,z)$ even for bounded sets $B\subset\R^n$, unlike in the study of local functionals where such a behaviour does not get along with the finiteness of the functional. Since we do not want to deal with such divergencies, we give here additionally a stronger condition on $f$.

\begin{definition}\label{defPBounded}
If the function $f:X\times X\times\R^n\times\R^n\to\R$ fulfils that there exist for every $M\in(0,\infty)$ a constant $C\in(0,\infty)$ and positive functions $\alpha_M\in L^1(X\times X)$ and $\beta_M\in L^1(X)$ with
\begin{equation}\label{eqIntegrabilityConditionSimplified}
|f(x,y,w,z)|\le\alpha_M(x,y)+\beta_M(x)\p(z)+\beta_M(y)\p(w)+C\p(w)\p(z)
\end{equation}
for almost all $(x,y)\in X\times X$ and all $w,z\in\R^n$, then we call $f$ a $p$-bounded function.
\end{definition}

\begin{corollary}\label{thIntegrabilityConditionSimplified}
Every $p$-bounded, measurable function $f:X\times X\times\R^n\times\R^n\to[0,\infty)$ fulfils
\begin{equation}\label{eqIntegrability1}
\int_X\int_Xf(x,y,u(x),u(y))\d x\d y < \infty\quad\text{for all}\quad u\in L^p(X;\R^n).
\end{equation}
\end{corollary}
\begin{proof}
We may directly verify condition~\eqref{eqIntegrability1}. Since the function $f$ is $p$-bounded, we find for every $M\in(0,\infty)$ a constant $C\in(0,\infty)$ and positive functions $\alpha_M\in L^1(X\times X)$ and $\beta_M\in L^1(X)$ such that~\eqref{eqIntegrabilityConditionSimplified} holds. Then we get with the property~\eqref{eqPropertyOfP} of the function $\p$ for every $u\in L^p(X;\R^n)$ that
\[ \int_X\int_Xf(x,y,u(x),u(y))\d x\d y\le\|\alpha_M\|_1+2\|\beta_M\|_1\p(\|u\|_p)+C\p(\|u\|_p)^2<\infty \]
if we choose in the case $p=\infty$ the constant $M$ greater than $\|u\|_\infty$. 
\end{proof}

To illustrate what kind of functions we are excluding with this stronger condition, let us construct an example of a function $f$ which is not $p$-bounded, but fulfils the integrability condition~\eqref{eqIntegrability1}.
\begin{example}\label{expDivergencies}
Let $X=(0,1)$ and choose $n=1$. We define the function $f$ by
\[ f:X\times X\times\R\times\R\to[0,\infty),\quad f(x,y,w,z)=\begin{cases}\frac1z&\text{if}\quad z\in[x,1],\\0&\text{otherwise.}\end{cases} \]
Then we have 
\begin{equation}\label{eqDivergentFunctional}
\int_X\int_Xf(x,y,\varphi(x),\psi(y))\d x\d y = \int_{\psi^{-1}((0,1])}\int_0^{\psi(y)}\frac1{\psi(y)}\d x\d y = \L1\big(\psi^{-1}((0,1])\big)
\end{equation}
for all functions $\varphi,\psi\in L^p(X)$. In particular, $f$ fulfils condition~\eqref{eqIntegrability1}.

On the other hand, $f$ cannot be $p$-bounded, since e.g. 
\[ \sup_{z\in[-1,1]}f(x,y,w,z)=\frac1x\quad\text{for all}\quad x,y\in X,\;w\in\R, \] 
which as a function of $x\in X$ is not integrable as an estimate of the form~\eqref{eqIntegrabilityConditionSimplified} would require.
\end{example}

However, this construction only works if the function $f$ depends on the variables $x$ and~$y$. Otherwise, every divergency of $f$ at finite values of $w$ and~$z$ would lead to a non-integrable divergency of $(x,y)\mapsto f(x,y,u(x),u(y))$ for some function $u\in L^p(X;\R^n)$. In fact, if $f$ does not explicitly depend on $x$ and~$y$, then the $p$-boundedness of~$f$ is equivalent to the integrability condition~\eqref{eqIntegrability1}.

\begin{proposition}
Let $f:\R^n\times\R^n\to[0,\infty)$ be a measurable function. Then we have
\begin{equation}\label{eqIntegrabilityHomogeneousCase}
\int_X\int_Xf(u(x),u(y))\d x\d y < \infty\quad\text{for all}\quad u\in L^p(X;\R^n)
\end{equation}
if and only if there exists for every $M\in(0,\infty)$ a constant $C_M\in(0,\infty)$ such that
\begin{equation}\label{eqIntegrabilityConditionHomogeneousCase}
f(w,z)\le C_M(1+\p(w))(1+\p(z))\quad\text{for all}\quad w,z\in\R^n.
\end{equation}
\end{proposition}

\begin{proof}
If $f$ fulfils the condition~\eqref{eqIntegrabilityConditionHomogeneousCase}, then~\eqref{eqIntegrabilityHomogeneousCase} holds by Corollary~\ref{thIntegrabilityConditionSimplified}.

For the other direction, we assume that for some $M\in(0,\infty)$, there does not exist a constant $C_M$ such that condition~\eqref{eqIntegrabilityConditionHomogeneousCase} holds. Then we find a sequence $\big((w_k,z_k)\big)_{k=1}^\infty\subset\R^n\times\R^n$ with 
\[ \frac{f(w_k,z_k)}{(1+\p(w_k))(1+\p(z_k))}\ge 2^{2k+2}\quad\text{for all}\quad k\in\N. \]
We choose pairwise disjoint subsets $E_k,F_k\subset X$, $k\in\N$, with
\[ \L m(E_k)=\frac{\L m(X)}{2^{k+1}(1+\p(w_k))}\quad\text{and}\quad\L m(F_k)=\frac{\L m(X)}{2^{k+1}(1+\p(z_k))}. \]
Such subsets exist since $\sum_{k=1}^\infty(\L m(E_k)+\L m(F_k))\le\L m(X)$ and since the Lebesgue measure is nonatomic, see e.g.~Corollary $1.21$ in~\cite{FonsecaLeoni2007}.

We now define the function
\[ u:X\to\R^n,\quad u(x)=\sum_{k=1}^\infty(w_k\chi_{E_k}(x)+z_k\chi_{F_k}(x)). \]
Then $u\in L^p(X;\R^n)$, since we have for $p\in[1,\infty)$
\[ \|u\|_p^p = \sum_{k=1}^\infty\big(|w_k|^p\L m(E_k)+|z_k|^p\L m(F_k)\big)\le\sum_{k=1}^\infty\frac{\L m(X)}{2^k}=\L m(X) \]
and for $p=\infty$ 
\[ \|u\|_\infty = \max\Big\{\sup_{k\in\N}|w_k|,\sup_{k\in\N}|z_k|\Big\}\le M. \]
Moreover, we find that
\begin{align*}
\int_X\int_Xf(u(x),u(y))\d x\d y&\ge\sum_{k=1}^\infty f(w_k,z_k)\L m(E_k)\L m(F_k) \\
&=\sum_{k=1}^\infty\frac{\L m(X)^2f(w_k,z_k)}{2^{2k+2}(1+\p(w_k))(1+\p(z_k))}=\infty.
\end{align*}
Thus, $f$ does not fulfil condition~\eqref{eqIntegrabilityHomogeneousCase}.
\end{proof}

\section{Strong Sequential Lower Semi-Continuity}\label{seStrongSemiContinuity}
Before we analyse the sequential lower semi-continuity of non-local functionals with respect to the weak or to the weak-star topology on $L^p(X;\R^n)$, we shortly give a criterion for the sequential lower semi-continuity with respect to the norm topology. To start with, let us briefly recall the definition of sequential lower semi-continuity.
\begin{definition}
Let $V$ be a topological space. Then a functional $\mathcal J:V\to\Rinf$ is called sequentially lower semi-continuous if we have for every sequence $(u_k)_{k\in\N}\subset V$ converging to $u\in V$ that
\[ \liminf_{k\to\infty}\mathcal J(u_k)\ge\mathcal J(u). \]
\end{definition}

Similar to the case of local functionals 
\[ \mathcal J_{\mathrm{local}}:L^p(X;\R^n)\to\Rinf,\quad\mathcal J_{\mathrm{local}}(u)=\int_Xf_{\mathrm{local}}(x,u(x))\d x, \]
where the sequential lower semi-continuity of the functional $\mathcal J_{\mathrm{local}}$ is equivalent to the lower semi-continuity of the function $\R^n\to\R$, $w\mapsto f_{\mathrm{local}}(x,w)$ for almost all $x\in X$, see e.g.\ Theorem~$5.9$ in~\cite{FonsecaLeoni2007}, the lower semi-continuity of the function $\R^n\times\R^n\to\R$, $(w,z)\mapsto f(x,y,w,z)$ is sufficient to guarantee the sequential lower semi-continuity of the non-local functional $\J$ provided the negative part of $f$ is additionally $p$-bounded. In the local case, this kind of lower bound was already necessary for the functional $\mathcal J_{\mathrm{local}}$ to be well-defined.

To simplify the notation, we define for every function $f:X\times X\times\R^n\times\R^n\to\R$ and every $(x,y)\in X\times X$ the function $f_{(x,y)}:\R^n\times\R^n\to\R$ by
\begin{equation}\label{eqfxy}
f_{(x,y)}(w,z)=f(x,y,w,z)\quad\text{for all}\quad w,z\in\R^n.
\end{equation}

\begin{proposition}\label{thStronglySemiContinuous}
Let $f:X\times X\times\R^n\times\R^n\to\R$ be a pairwise symmetric, measurable function whose negative part is $p$-bounded.

Then the non-local functional $\J$ on $L^p(X;\R^n)$ defined by the function $f$ is sequentially lower semi-continuous with respect to the strong topology on $L^p(X;\R^n)$ if the function $f_{(x,y)}$ is for almost all $(x,y)\in X\times X$ lower semi-continuous.
\end{proposition}

\begin{proof}
Let $(u_k)_{k\in\N}\subset L^p(X;\R^n)$ be a sequence converging to $u\in L^p(X;\R^n)$. We choose a subsequence $(u_{k_\ell})_{\ell\in\N}$ of $(u_k)_{k\in\N}$ such that
\[ \lim_{\ell\to\infty}\J(u_{k_\ell}) = \liminf_{k\to\infty}\J(u_k) \]
holds and such that we have
\[ \lim_{\ell\to\infty}u_{k_\ell}(x) = u(x)\quad\text{for almost all}\quad x\in X. \]

To be able to apply Fatou's lemma, we use that the negative part of $f$ is $p$-bounded. We thus find for every $M\in(0,\infty)$ a constant $C\in(0,\infty)$ and positive functions $\alpha_M\in L^1(X\times X)$ and $\beta_M\in L^1(X)$ such that
\[ f(x,y,w,z)+\alpha_M(x,y)+\beta_M(x)\p(z)+\beta_M(y)\p(w)+C\p(w)\p(z) \ge 0 \]
for almost all $(x,y)\in X\times X$ and all $w,z\in\R^n$.
We then choose the constant $M$ in the case $p=\infty$ greater than $\sup_{\ell\in\N}\|u_{k_\ell}\|_\infty$ and find with the lower semi-continuity of $f_{(x,y)}$ for almost all $(x,y)\in X\times X$ that
\begin{align*}
&\lim_{\ell\to\infty}\int_X\int_Xf(x,y,u_{k_\ell}(x),u_{k_\ell}(y))\d x\d y+\|\alpha_M\|_1+2\|\beta_M\|_1\p(\|u\|_p)+C\p(\|u\|_p)^2 \\
&\qquad\ge\int_X\int_X\liminf_{\ell\to\infty}\big(f(x,y,u_{k_\ell}(x),u_{k_\ell}(y))+\alpha_M(x,y)+\beta_M(x)\p(u_{k_\ell}(y)) \\
&\qquad\qquad\qquad\qquad+\beta_M(y)\p(u_{k_\ell}(x))+C\p(u_{k_\ell}(x))\p(u_{k_\ell}(y))\big)\d x\d y \\
&\qquad\ge\int_X\int_Xf(x,y,u(x),u(y))\d x\d y+\|\alpha_M\|_1+2\|\beta_M\|_1\p(\|u\|_p)+C\p(\|u\|_p)^2.
\end{align*}
Thus, $\liminf_{k\to\infty}\J(u_k)\ge\J(u)$, and we conclude that $\J$ is sequentially lower semi-continuous with respect to the strong topology on $L^p(X;\R^n)$.
\end{proof}

We remark that without the lower bound on the function $f$, the lower semi-continuity of the function $f_{(x,y)}$ for almost all $(x,y)\in X\times X$ does not imply the sequential lower semi-continuity of the non-local functional $\J$.
\begin{example}
Similar to Example~\ref{expDivergencies}, we choose $X=(0,1)$ and $n=1$ and define the map
\[ \tilde f:X\times\R\to\R,\quad\tilde f(x,z)=\begin{cases}-\frac1z&\text{if}\quad z\in[x,1],\\\phantom-0&\text{otherwise.}\end{cases} \]
Then the function
\[ f:X\times X\times\R\times\R\to\R,\quad f(x,y,w,z)=\tfrac12(\tilde f(x,z)+\tilde f(y,w)) \]
has a negative part which is not $p$-bounded as was shown in Example~\ref{expDivergencies}, but $f_{(x,y)}$ is for all $x,y\in X$ lower semi-continuous.

On the other hand, we find from~\eqref{eqDivergentFunctional} for the non-local functional $\J$ defined by the function $f$ that
\[ \J(u) = -\L1\big(u^{-1}((0,1])\big). \]
So, for the sequence $(u_k)_{k\in\N}\subset L^p(X;\R^n)$ defined by $u_k(x)=\frac1k$ for all $x\in X$, $k\in\N$, we get that $(u_k)_{k\in\N}$ converges uniformly to the zero function, but
\[ \liminf_{k\to\infty}\J(u_k) = -1 < 0 = \J(0). \]
Thus, $\J$ is not sequentially lower semi-continuous with respect to the strong topology on $L^p(X;\R^n)$.
\end{example}

\section{Weak Sequential Lower Semi-Continuity}\label{seWeakSemiContinuity}
After all the preparations, we are now ready to study the sequential lower semi-continuity of non-local functionals with respect to the weak topology on $L^p(X;\R^n)$ for $p\in[1,\infty)$ and with respect to the weak-star topology on $L^\infty(X;\R^n)$ for $p=\infty$.

\begin{theorem}\label{thCharacterisationOfWeakSlsc}
Let $f:X\times X\times\R^n\times\R^n\to\R$ be a pairwise symmetric, measurable function whose negative part is $p$-bounded.
Moreover, we assume that the function $f_{(x,y)}$, defined by~\eqref{eqfxy}, is continuous for almost all $(x,y)\in X\times X$ and that there exist for every $M\in(0,\infty)$ positive functions $\alpha_M\in L^1(X\times X)$ and $\beta_M\in L^1(X)$ such that
\begin{equation}\label{eqUpperBound}
f(x,y,w,z)\le\alpha_M(x,y)+\beta_M(x)\p(z)
\end{equation}
for almost all $(x,y)\in X\times X$ and all $w,z\in\R^n$ with $|w|\le M$.

Then the non-local functional $\J$ on $L^p(X;\R^n)$ defined by the function $f$ is sequentially lower semi-continuous with respect to the weak topology on $L^p(X;\R^n)$ for $p\in[1,\infty)$ and with respect to the weak-star topology on $L^\infty(X;\R^n)$ for $p=\infty$ if and only if the function
\begin{equation}\label{eqDefinitionOfPhi}
\Phi_{x,\psi}:\R^n\to\R,\quad\Phi_{x,\psi}(w)=\int_Xf(x,y,w,\psi(y))\d y
\end{equation}
is for every $\psi\in L^p(X;\R^n)$ for almost all $x\in X$ convex.
\end{theorem}

\begin{proof}
We first show that the function $\Phi_{x,\psi}$ is for every $\psi\in L^p(X;\R^n)$ for almost all $x\in X$ convex if $\J$ is sequentially lower semi-continuous. 

So, let $\psi\in L^p(X;\R^n)$. To begin with, we assume that $\J(\psi)<\infty$. 
Since the Lebesgue measure is nonatomic, we find for every $\vartheta\in[0,1]$ a sequence $(E_{\vartheta,k})_{k\in\N}$ of subsets of $X$ such that the characteristic functions $\chi_{E_{\vartheta,k}}\in L^\infty(X)$, $k\in\N$, converge weakly-star in $L^\infty(X)$ to the constant function $\vartheta$, see e.g.\ Proposition~$2.87$ in~\cite{FonsecaLeoni2007}. 
Then we define for arbitrary functions $\omega_1,\omega_2\in L^\infty(X;\R^n)$, $\vartheta\in[0,1]$, and an arbitrary measurable subset $A\subset X$ the functions $u_k\in L^p(X;\R^n)$,
\[ u_k(x) = \chi_A(x)\varphi_k(x)+\chi_{A\c}(x)\psi(x),\quad k\in\N, \]
where we use the notation $A\c=X\setminus A$ and where the functions $\varphi_k\in L^\infty(X;\R^n)$ are given by
\[ \varphi_k(x)=\chi_{E_{\vartheta,k}}(x)\,\omega_1(x)+(1-\chi_{E_{\vartheta,k}}(x))\,\omega_2(x),\quad k\in\N. \]

The sequence $(u_k)_{k\in\N}$ thus converges weakly in $L^p(X;\R^n)$ if $p\in[1,\infty)$ and weakly-star in $L^\infty(X;\R^n)$ if $p=\infty$ to the function $u\in L^p(X;\R^n)$ defined by
\[ u(x) = \chi_A(x)\bar\omega(x)+\chi_{A\c}(x)\psi(x) \]
where $\bar\omega\in L^\infty(X;\R^n)$ is the convex combination $\bar\omega=\vartheta\omega_1+(1-\vartheta)\omega_2$ of $\omega_1$ and $\omega_2$. So, the sequential lower semi-continuity of $\J$ implies that $\liminf_{k\to\infty}\J(u_k)\ge\J(u)$. We therefore get for all measurable sets $A\subset X$ the inequality
\begin{multline}\label{eqInsertingCharacteristicFunctions}
\liminf_{k\to\infty}\left(\int_A\int_Af(x,y,\varphi_k(x),\varphi_k(y))\d y\d x+2\int_A\int_{A\c}f(x,y,\varphi_k(x),\psi(y))\d y\d x\right) \\
\ge\int_A\int_Af(x,y,\bar\omega(x),\bar\omega(y))\d y\d x+2\int_A\int_{A\c}f(x,y,\bar\omega(x),\psi(y))\d y\d x.
\end{multline}

To get rid of the integrals over $A\times A$, we will now consider the limit where the measure of $A$ tends to zero. 
We remark that because of the $p$-boundedness of $f^-$ and the upper bound~\eqref{eqUpperBound} of $f$, there exists for every $M\in(0,\infty)$ a function $g_M\in L^1(X\times X)$ such that we have
\[ |f(x,y,\omega(x),\tilde\omega(y))|\le g_M(x,y)\quad\text{and}\quad |f(x,y,\omega(x),\psi(y))|\le g_M(x,y) \]
for all $\omega,\tilde\omega\in L^\infty(X;\R^n)$ with $\|\omega\|_\infty\le M$, $\|\tilde\omega\|_\infty\le M$ and almost all $(x,y)\in X\times X$.
In particular, we have with $M\ge\max\{\|\omega_1\|_\infty,\|\omega_2\|_\infty\}$ that
\begin{gather*}
|f(x,y,\varphi_k(x),\varphi_k(y))|\le\max_{i,j\in\{1,2\}}|f(x,y,\omega_i(x),\omega_j(y))|\le g_M(x,y), \\
|f(x,y,\bar\omega(x),\bar\omega(y))|\le g_M(x,y)
\end{gather*}
for almost all $(x,y)\in X\times X$ and all $k\in\N$.
To get a bound for the integrals over $A\times A$, we choose for every $\varepsilon\in(0,\infty)$ and every measurable set $E\subset X$ with positive measure a set $E_M'\subset E$ with positive measure such that
\[ \int_A\int_Ag_M(x,y)\d x\d y\le\varepsilon\L m(A). \]
for all measurable sets $A\subset E_M'$.
We then get from inequality~\eqref{eqInsertingCharacteristicFunctions} that
\begin{equation}\label{eqIntermediateStepToConvexity}
\liminf_{k\to\infty}\int_A\int_{A\c}f(x,y,\varphi_k(x),\psi(y))\d y\d x+\varepsilon\L m(A)\ge\int_A\int_{A\c}f(x,y,\bar\omega(x),\psi(y))\d y\d x
\end{equation}
for all measurable sets $A\subset E_M'$.

By definition of the functions $\varphi_k$, we find for all measurable sets $A\subset E_M'$ and all $k\in\N$ that
\begin{multline*}
\int_A\int_{A\c}f(x,y,\varphi_k(x),\psi(y))\d y\d x = \int_A\chi_{E_{\vartheta,k}}(x)\int_{A\c}f(x,y,\omega_1(x),\psi(y))\d y\d x \\
+\int_A(1-\chi_{E_{\vartheta,k}}(x))\int_{A\c}f(x,y,\omega_2(x),\psi(y))\d y\d x.
\end{multline*}
Since the functions $\chi_{E_{\vartheta,k}}$ converge for $k\to\infty$ by definition of the sets $E_{\vartheta,k}$ weakly-star in $L^\infty(X)$ to the constant function $\vartheta$, we further get that
\begin{multline}\label{eqLiminfEvaluated}
\liminf_{k\to\infty}\int_A\int_{A\c}f(x,y,\varphi_k(x),\psi(y))\d y\d x = \vartheta\int_A\int_{A\c}(f(x,y,\omega_1(x),\psi(y))\d y\d x\\
+(1-\vartheta)\int_A\int_{A\c}f(x,y,\omega_2(x),\psi(y))\d y\d x
\end{multline}
for all measurable sets $A\subset E_M'$. 

By our choice of the function $g_M$ and the set $E_M'$ we have that
\[ \int_A\int_A|f(x,y,\omega(x),\psi(y))|\d y\d x\le\varepsilon\L m(A) \]
for all measurable sets $A\subset E_M'$ and all functions $\omega\in L^\infty(X;\R^n)$ with $\|\omega\|_\infty\le M$.
Therefore, we get by plugging~\eqref{eqLiminfEvaluated} into~\eqref{eqIntermediateStepToConvexity} the inequality
\[ \int_A\big(\vartheta\Phi_{x,\psi}(\omega_1(x))+(1-\vartheta)\Phi_{x,\psi}(\omega_2(x))\big)\d x+3\varepsilon\L m(A)\ge\int_A\Phi_{x,\psi}(\bar\omega(x))\d x \]
for all measurable sets $A\subset E_M'$. Since this holds for every $M\in(0,\infty)$ for all functions $\omega_1,\omega_2\in L^\infty(X;\R^n)$ with $\|\omega_1\|_\infty\le M$ and $\|\omega_2\|_\infty\le M$, we can use Lemma~\ref{thFromIntegralToPointwise} and finally get after letting $\varepsilon$ tend to zero that
\[ \vartheta\Phi_{x,\psi}(w_1)+(1-\vartheta)\Phi_{x,\psi}(w_2) \ge \Phi_{x,\psi}(\vartheta w_1+(1-\vartheta)w_2) \]
for almost all $x\in X$, all $\vartheta\in[0,1]$, and all $w_1,w_2\in\R^n$, which proves the convexity of $\Phi_{x,\psi}$ for almost all $x\in X$.

It remains to consider the case where $\J(\psi)=\infty$. Here we use that the $p$-boundedness of $f^-$ and the upper bound~\eqref{eqUpperBound} of $f$ ensure that $\J(\omega)<\infty$ for all $\omega\in L^\infty(X;\R^n)$. We choose a sequence $(\psi_k)_{k\in\N}\subset L^\infty(X;\R^n)$ which converges pointwise almost everywhere and strongly in $L^p(X;\R^n)$ to the function $\psi$. Then, by the previous result, we know that the functions $\Phi_{x,\psi_k}$, $k\in\N$, are for almost all $x\in X$ convex. Moreover, the local functional 
\[ L^p(X;\R^n)\to\Rinf,\quad v\mapsto\Phi_{x,v}(w) = \int_Xf(x,y,w,v(y))\d y \]
is -- because of the continuity of the function $f_{(x,y)}$ for almost all $(x,y)\in X\times X$, the $p$-boundedness of $f^-$, and the upper bound~\eqref{eqUpperBound} of~$f$ -- for almost all $x\in X$ and all $w\in\R^n$ continuous with respect to the strong topology, see e.g.\ Corollaries~$6.51$ and~$6.53$ in~\cite{FonsecaLeoni2007} (the proof works in the same way as the proof of Proposition~\ref{thStronglySemiContinuous}). Therefore, we have for almost all $x\in X$ that
\[ \Phi_{x,\psi}(w) = \lim_{k\to\infty}\Phi_{x,\psi_k}(w), \]
which shows that the function $\Phi_{x,\psi}$ is for almost all $x\in X$ convex.

This concludes the proof that the sequential lower semi-continuity of $\J$ implies for every function $\psi\in L^p(X;\R^n)$ for almost all $x\in X$ the convexity of the function $\Phi_{x,\psi}$.

\medskip
For the other direction, we assume that the function $\Phi_{x,\psi}$ is for every $\psi\in L^p(X;\R^n)$ for almost all $x\in X$ convex. Let $(u_k)_{k\in\N}\subset L^p(X;\R^n)$ be a sequence converging weakly in $L^p(X;\R^n)$ if $p\in[1,\infty)$ and weakly-star in $L^\infty(X;\R^n)$ if $p=\infty$ to a function $u\in L^p(X;\R^n)$. In particular, $(u_k)_{k\in\N}$ is bounded in $L^p(X;\R^n)$ and therefore, there exists a subsequence $(u_{k_\ell})_{\ell\in\N}$ of $(u_k)_{k\in\N}$ generating a Young measure $\nu$. 

I.e.\ we have a map $\nu:X\to\mathcal M(\R^n;\R)$, $x\mapsto\nu_x$, where $\mathcal M(\R^n;\R)$ denotes all signed Radon measures on~$\R^n$, which fulfils that $\nu_x$ is a probability measure for almost all $x\in X$, and that for all $\phi\in C_0(\R^n)$ the function
$X\to\R$, $x\mapsto\int_{\R^n}\phi(w)\d\nu_x(w)$
is  measurable and satisfies for every $h\in L^1(X)$ the equality
\begin{equation}\label{eqDefinitionOfYoungMeasure} 
\lim_{\ell\to\infty}\int_Xh(x)\phi(u_{k_\ell}(x))\d x = \int_Xh(x)\int_{\R^n}\phi(w)\d\nu_x(w)\d x.
\end{equation}
For a detailed introduction into the theory of Young measures, we refer to Chapter $8$ in~\cite{FonsecaLeoni2007}.

Since $(u_{k_\ell})_{\ell\in\N}$ generates the Young measure $\nu$, the sequence $(u_{k_\ell}\times u_{k_\ell})_{\ell\in\N}$ in the space $L^p(X\times X;\R^n\times\R^n)$ defined by $(u_{k_\ell}\times u_{k_\ell})(x,y) = (u_{k_\ell}(x),u_{k_\ell}(y))$ generates the Young measure $\nu\otimes\nu:X\times X\to\mathcal M(\R^n\times\R^n;\R)$ defined by $(\nu\otimes\nu)_{(x,y)}=\nu_x\otimes\nu_y$, where $\nu_x\otimes\nu_y$ denotes the product measure of $\nu_x$ and $\nu_y$.
Indeed, using the Stone--Weierstra\ss\ theorem, it is enough to verify that
\begin{multline*}
\lim_{\ell\to\infty}\int_X\int_Xh(x,y)\phi(u_{k_\ell}(x),u_{k_\ell}(y))\d x\d y \\
=\int_X\int_Xh(x,y)\int_{\R^n}\int_{\R^n}\phi(w,z)\d\nu_x(w)\d\nu_y(z)\d x\d y
\end{multline*}
holds for all functions $h\in L^1(X\times X)$, $\phi\in C_0(\R^n\times\R^n)$ of the form $h(x,y)=h_1(x)h_2(y)$ and $\phi(w,z)=\phi_1(w)\phi_2(z)$, $h_1,h_2\in L^1(X)$, $\phi_1,\phi_2\in C_0(\R^n)$. But this directly follows from Fubini's theorem and the relation~\eqref{eqDefinitionOfYoungMeasure}, see Proposition~$2.3$ in~\cite{Pedregal1997}.

Now, the $p$-boundedness of $f^-$ implies by the Dunford--Pettis theorem that the functions
\[ X\times X\to[0,\infty),\quad (x,y)\mapsto f^-(x,y,u_{k_\ell}(x),u_{k_\ell}(y)),\quad\ell\in\N, \]
are uniformly integrable. Thus, we can apply the fundamental theorem for Young measures, see e.g.\ Theorem~$8.6$ in~\cite{FonsecaLeoni2007}, and find that
\begin{equation}\label{eqLiminfInequality1}
\liminf_{\ell\to\infty}\J(u_{k_\ell})\ge\int_X\int_X\int_{\R^n}\int_{\R^n}f(x,y,w,z)\d\nu_y(z)\d\nu_x(w)\d y\d x.
\end{equation}

Moreover, condition~\eqref{eqUpperBound} implies that for almost all $x\in X$ and all $w\in\R^n$ the functions
\[ X\to\R,\quad y\mapsto f(x,y,w,u_{k_\ell}(y)),\quad\ell\in\N, \]
are uniformly integrable. Thus, we get from the continuity of the functions $f_{(x,y)}$ for almost all $(x,y)\in X\times X$ again with the fundamental theorem for Young measures that
\[ \lim_{\ell\to\infty}\int_Xf(x,y,w,u_{k_\ell}(y))\d y = \int_X\int_{\R^n}f(x,y,w,z)\d\nu_y(z)\d y \]
for almost all $x\in X$ and all $w\in\R^n$. So for almost all $x\in X$, the function
\begin{equation}\label{eqDefinitionOfPhiTilde}
\tilde\Phi_{x,\nu}:\R^n\to\Rinf,\quad\tilde\Phi_{x,\nu}(w)=\int_X\int_{\R^n}f(x,y,w,z)\d\nu_y(z)\d y
\end{equation}
is the limit of the convex functions $\Phi_{x,u_{k_\ell}}$, $\ell\in\N$, and is therefore convex.

Using now that $\nu_x$ is by definition of a Young measure for almost all $x\in X$ a probability measure, we find with Jensen's inequality that
\begin{align}
\int_X\int_X\int_{\R^n}\int_{\R^n}f(x,y,w,z)\d\nu_y(z)\d\nu_x(w)\d y\d x&=\int_X\int_{\R^n}\tilde\Phi_{x,\nu}(w)\d\nu_x(w)\d x \nonumber \\
&\ge\int_X\tilde\Phi_{x,\nu}\big({\textstyle\int_{\R^n}w\d\nu_x(w)}\big)\d x. \label{eqLiminfInequality2}
\end{align}

Since the sequence $(u_{k_\ell})_{\ell\in\N}$ converges weakly in $L^p(X;\R^n)$ if $p\in[1,\infty)$ and weakly-star in $L^\infty(X;\R^n)$ if $p=\infty$ to $u$, relation~\eqref{eqDefinitionOfYoungMeasure} implies that
\[ \int_{\R^n}w\d\nu_x(w) = u(x)\quad\text{for almost all}\quad x\in X. \]
Together with the pairwise symmetry of~$f$, we then get
\begin{equation}\label{eqLiminfInequality3}
\int_X\tilde\Phi_{x,\nu}({\textstyle\int_{\R^n}w\d\nu_x(w)})\d x = \int_X\tilde\Phi_{x,\nu}(u(x))\d x = \int_X\int_{\R^n}\Phi_{y,u}(z)\d\nu_y(z)\d y.
\end{equation}

Using now the convexity of $\Phi_{y,u}$ for almost all $y\in X$, we get again with Jensen's inequality that
\begin{equation}\label{eqLiminfInequality4}
\int_X\int_{\R^n}\Phi_{y,u}(z)\d\nu_y(z)\d y \ge \int_X\Phi_{y,u}\big({\textstyle\int_{\R^n}z\d\nu_y(z)}\big)\d y = \int_X\Phi_{y,u}(u(y))\d y = \J(u).
\end{equation}
Putting together the inequalities~\eqref{eqLiminfInequality1}, \eqref{eqLiminfInequality2}, \eqref{eqLiminfInequality3}, and \eqref{eqLiminfInequality4}, we have shown that
\[ \liminf_{\ell\to\infty}\J(u_{k_\ell}) \ge \J(u), \]
proving the \lsC\ of $\J$.
\end{proof}

Here, the proof that the convexity of the functions $\Phi_{x,\psi}$ implies the sequential lower semi-continuity of the functional~$\J$ makes only use of the upper bound~\eqref{eqUpperBound} to show the convexity of the function~$\tilde\Phi_{x,\nu}$ defined in~\eqref{eqDefinitionOfPhiTilde}. We can therefore waive this upper bound if we guarantee the convexity of~$\tilde\Phi_{x,\nu}$ by imposing that the function~$f_{(x,y)}$ is separately convex (i.e.\ the maps $\R^n\to\R$, $w\mapsto f_{(x,y)}(w,z)$ and $\R^n\to\R$, $w\mapsto f_{(x,y)}(z,w)$ are convex for all $z\in\R^n$) for almost all $(x,y)\in X\times X$, see~\cite{Pedregal1997}.

\begin{corollary}
Let $f:X\times X\times\R^n\times\R^n\to\R$ be a pairwise symmetric, measurable function whose negative part is $p$-bounded.

Then the non-local functional $\J$ on $L^p(X;\R^n)$ defined by the function~$f$ is sequentially lower semi-continuous with respect to the weak topology on $L^p(X;\R^n)$ for $p\in[1,\infty)$ and with respect to the weak-star topology on $L^\infty(X;\R^n)$ for $p=\infty$ if the function $f_{(x,y)}$, defined by~\eqref{eqfxy}, is for almost all $(x,y)\in X\times X$ separately convex.
\end{corollary}

\section{Equivalent Integrands}\label{seEquivalentIntegrands}
In this section, we will try to characterise the classes of functions which define the same non-local functional. In particular, we are interested in finding a good representative for each of these classes and thereby to possibly simplify the criterion of sequential lower semi-continuity given in Theorem~\ref{thCharacterisationOfWeakSlsc}.

We will restrict our attention to rather regular integrands.
\begin{definition}
Let $f:X\times X\times\R^n\times\R^n\to\R$ be a pairwise symmetric, measurable function. Moreover, we assume that the function 
\begin{equation}\label{eqBaseIntegrand}
X\times X\to\R,\quad (x,y)\mapsto f(x,y,0,0)
\end{equation}
is integrable, that the function $f_{(x,y)}$, defined by~\eqref{eqfxy}, is continuously differentiable for almost all $(x,y)\in X\times X$, and that there exist for every $M\in(0,\infty)$ positive functions $\alpha_M\in L^{p^*}(X)\otimes L^1(X)$ and $\beta_M\in L^1(X)$, with $p^*$ being the H\"older conjugate of $p$, such that
\begin{equation}\label{eqBoundForFirstDerivative}
|\nabla_wf(x,y,w,z)|\le\alpha_M(x,y)+\beta_M(y)\pp{p-1}M(w)
\end{equation}
for almost all $(x,y)\in X\times X$ and all $w,z\in\R^n$ with $|z|\le M$. 
Then we call $f$ a $p$-regular function.
\end{definition}

We remark that for a $p$-regular function $f$ the estimate
\[ |f(x,y,w,z)| \le |f(x,y,0,0)|+\int_0^1|w||\nabla_wf(x,y,tw,0)|\d t+\int_0^1|z||\nabla_zf(x,y,w,tz)|\d t \]
implies together with the integrability of the function~\eqref{eqBaseIntegrand} and the bound~\eqref{eqBoundForFirstDerivative} for the derivative of $f$ that there exist for every $M\in(0,\infty)$ positive functions $\alpha_M\in L^1(X\times X)$ and $\beta_M\in L^1(X)$ such that
\[ |f(x,y,w,z)|\le\alpha_M(x,y)+\beta_M(x)\p(z) \]
for almost all $(x,y)\in X\times X$ and all $w,z\in\R^n$ with $|w|\le M$. 

\medskip
We will in the following give a characterisation of the class of functions whose corresponding non-local functional constantly vanishes.
If we only consider real-valued non-local functionals, then two functions define the same non-local functional if and only if they differ by a function of this class.
\begin{definition}
We denote by $\mathcal N^p$ the set of all pairwise symmetric, measurable functions $f:X\times X\times\R^n\times\R^n\to\R$ whose negative part $f^-$ obeys condition~\eqref{eqIntegrabilityNegativePart} and for which the non-local functional $\J$ on $L^p(X;\R^n)$ defined by $f$ fulfils $\J(u)=0$ for all $u\in L^p(X;\R^n)$.
\end{definition}
We further introduce a subset~$\mathcal N_0^p$ of~$\mathcal N^p$ which is easier to parametrise.
\begin{definition}
Let $\mathcal N_0^p$ denote the set of all pairwise symmetric, measurable functions $f:X\times X\times\R^n\times\R^n\to\R$ for which there exist a measurable function $g:X\times X\times\R^n\to\R$ and a symmetric function $h\in L^1(X\times X)$ with the properties that we find for every $M\in(0,\infty)$ a function $\alpha_M\in L^1(X)$ and a constant $C\in(0,\infty)$ with
\begin{equation}\label{eqUpperBoundForG}
\int_X|g(x,y,w)|\d y\le\alpha_M(x)+C\p(w)
\end{equation}
for almost all $x\in X$ and all $w\in\R^n$, 
\begin{equation}\label{eqMeanValueZero}
\int_X\int_Xh(x,y)\d x\d y = 0,\quad\int_Xg(x,y,w)\d y=0
\end{equation}
for almost all $x\in X$ and all $w\in\R^n$, and
\begin{equation}\label{eqZeroIntegrals}
f(x,y,w,z) = g(x,y,w)+g(y,x,z)+h(x,y)
\end{equation}
for almost all $(x,y)\in X\times X$ and all $w,z\in\R^n$.
\end{definition}
\begin{lemma}
We have $\mathcal N_0^p\subset\mathcal N^p$.
\end{lemma}
\begin{proof}
Let $f\in\mathcal N_0^p$. Then there exist a measurable function $g:X\times X\times\R^n\to\R$ and a symmetric function $h\in L^1(X\times X)$ with the properties~\eqref{eqUpperBoundForG}, \eqref{eqMeanValueZero}, and~\eqref{eqZeroIntegrals}. Because of the integrability condition~\eqref{eqUpperBoundForG}, we can use Fubini's theorem and find with~\eqref{eqMeanValueZero} that the non-local functional $\J$ defined by the function $f$ fulfils
\[ \J(u) = \int_X\int_Xg(x,y,u(x))\d y\d x + \int_X\int_Xg(y,x,u(y))\d x\d y = 0 \]
for every $u\in L^p(X;\R^n)$. Thus, $f\in\mathcal N^p$.
\end{proof}

If we restrict our attention to $p$-regular functions, then there is no difference between $\mathcal N_0^p$ and $\mathcal N^p$.

\begin{proposition}\label{thZeroFunctional}
Let $f:X\times X\times\R^n\times\R^n\to\R$ be a $p$-regular function. Then $f\in\mathcal N^p$ if and only if $f\in\mathcal N_0^p$.
\end{proposition}
\begin{proof}
Let $f\in\mathcal N^p$. Then the non-local functional $\J$ defined by $f$ is constantly equal to zero. So, we can take the variational derivative of $\J$ and get with the pairwise symmetry of the function~$f$ that
\begin{equation}\label{eqVariationalDerivativeIsZero}
0 = \lim_{t\to0}\frac{\J(u+tv)-\J(u)}t = 2\sum_{i=1}^n\int_X\int_X\partial_{w_i}f(x,y,u(x),u(y))v_i(x)\d x\d y
\end{equation}
for all functions $u,v\in L^\infty(X;\R^n)$. Here, we have used the bound~\eqref{eqBoundForFirstDerivative} of the first partial derivative of~$f$ to differentiate under the integral sign.

Thus, we have for all $i\in[1,n]\cap\N$, almost all $x\in X$, all $w\in\R^n$, and all functions $u\in L^\infty(X;\R^n)$ that
\begin{equation}\label{eqMeanOfDerivative}
\int_X\partial_{w_i}f(x,y,w,u(y))\d y = 0.
\end{equation}
Indeed, there would otherwise exist an $i_0\in[1,n]\cap\N$, a set $A\subset X$ with positive measure, a function $\varphi\in L^\infty(A;\R^n)$, a function $u\in L^\infty(X;\R^n)$, a constant $\delta\in(0,\infty)$, and a sign $\epsilon\in\{-1,1\}$ such that
\[ \epsilon\int_X\partial_{w_{i_0}}f(x,y,\varphi(x),u(y))\d y\ge\delta\quad\text{for all}\quad x\in A. \]
Using the bound~\eqref{eqBoundForFirstDerivative} of the function $\partial_{w_{i_0}}f$, we find a measurable subset $\tilde A\subset A$ such that
\[ \int_{\tilde A}\int_{\tilde A}|\partial_{w_{i_0}}f(x,y,\varphi(x),\psi(y))|\d y\d x<\frac\delta2\L m(\tilde A) \]
for all functions $\psi\in L^\infty(X;\R^n)$ with $\|\psi\|_\infty\le\max\{\|u\|_\infty,\|\varphi\|_\infty\}$.
Defining now the functions $\tilde u,v\in L^\infty(X;\R^n)$ by
\[ \tilde u(x) = \begin{cases}u(x)&\text{if}\quad x\in X\setminus\tilde A,\\\varphi(x)&\text{if}\quad x\in\tilde A\end{cases} \]
and $v_i(x)=\delta_{i,i_0}\chi_{\tilde A}(x)$, we get
\begin{multline*}
\epsilon\sum_{i=1}^n\int_X\int_X\partial_{w_i}f(x,y,\tilde u(x),\tilde u(y))v_i(x)\d x\d y
=\epsilon\int_{\tilde A}\int_X\partial_{w_{i_0}}f(x,y,\varphi(x),\tilde u(y))\d y\d x \\
>\epsilon\int_{\tilde A}\int_X\partial_{w_{i_0}}f(x,y,\varphi(x),u(y))\d y\d x-\delta\L m(\tilde A)\ge0,
\end{multline*}
which is a contradiction to~\eqref{eqVariationalDerivativeIsZero}.

\medskip
Now, equation~\eqref{eqMeanOfDerivative} is only possible if the function $\R^n\to\R$, $z\mapsto\partial_{w_i}f(x,y,w,z)$ is constant for all $i\in[1,n]\cap\N$, almost all $(x,y)\in X\times X$, and all $w\in\R^n$. Defining therefore the function $g:X\times X\times\R^n\to\R$ by
\[ g(x,y,w) = \sum_{i=1}^n\int_0^1w_i\partial_{w_i}f(x,y,tw,0)\d t, \]
we find with the fundamental theorem of calculus that $g(x,y,w)=f(x,y,w,z)-f(x,y,0,z)$ holds for almost all $(x,y)\in X\times X$ and all $w,z\in\R^n$. Thus, we get with the pairwise symmetry of $f$ that 
\[ f(x,y,w,z)=g(x,y,w)+g(y,x,z)+f(x,y,0,0) \]
for almost all $(x,y)\in X\times X$ and all $w,z\in\R^n$. So, $f$ has the form~\eqref{eqZeroIntegrals}, where the function $h\in L^1(X\times X)$ is defined by $h(x,y)=f(x,y,0,0)$ for all $x,y\in X$. Moreover, the $p$-regularity of $f$ implies the bound~\eqref{eqUpperBoundForG} for $g$, and the conditions~\eqref{eqMeanValueZero} finally follow from $\J(0)=0$ and from~\eqref{eqMeanOfDerivative} together with Fubini's theorem. Thus, we have shown that $f\in\mathcal N_0^p$ which concludes the proof.
\end{proof}

In the following, we will try to use this ambiguity in the integrand of a non-local functional to find for a non-local functional $\J$, which is sequentially lower semi-continuous with respect to the weak topology if $p\in[1,\infty)$ and to weak-star topology if $p=\infty$, an integrand $\tilde f$ defining the functional $\J$ such that $\tilde f_{(x,y)}$ is separately convex. But it seems that this is only possible in the case where the functional is defined on a Lebesgue space of real-valued functions, i.e. for $n=1$.

\begin{theorem}\label{thConvexityOfTheIntegrand}
Let $n=1$ and let $f:X\times X\times\R\times\R\to\R$ be a $p$-regular function which additionally fulfils that the function $f_{(x,y)}$, defined by~\eqref{eqfxy}, is for almost all $(x,y)\in X\times X$ two times continuously differentiable and that there exists for every $M\in(0,\infty)$ a function $\alpha_M\in L^1(X\times X)$ such that
\begin{equation}\label{eqBoundForSecondDerivative}
|\partial_w^2f(x,y,w,z)|\le\alpha_M(x,y)
\end{equation}
for almost all $(x,y)\in X\times X$ and all $w,z\in\R$ with $|w|\le M$ and $|z|\le M$.

Then the function
\[ \Phi_{x,\psi}:\R\to\R,\quad \Phi_{x,\psi}(w) = \int_Xf(x,y,w,\psi(y))\d y \] 
is for every function $\psi\in L^p(X)$ for almost all $x\in X$ convex if and only if there exist a pairwise symmetric, measurable function $\tilde f:X\times X\times\R\times\R\to\R$ and a function $f_0\in\mathcal N_0^p$ such that $\tilde f_{(x,y)}$ is for almost all $(x,y)\in X\times X$ separately convex and
\begin{equation}\label{eqEquivalenceClassOfSeparateConvexFunctions}
f(x,y,w,z) = \tilde f(x,y,w,z)+f_0(x,y,w,z)
\end{equation}
for almost all $(x,y)\in X\times X$ and all $w,z\in\R$.
\end{theorem}
\begin{proof}
Let us first assume that the function $f$ is of the form~\eqref{eqEquivalenceClassOfSeparateConvexFunctions}. Then there exist by definition of the set $\mathcal N_0^p$ a measurable function $g:X\times X\times\R\to\R$ and a symmetric function $h\in L^1(X\times X)$ with the properties~\eqref{eqUpperBoundForG} and~\eqref{eqMeanValueZero} such that
\[ f(x,y,w,z) = \tilde f(x,y,w,z)+g(x,y,w)+g(y,x,z)+h(x,y) \]
for almost all $(x,y)\in X\times X$ and all $w,z\in\R$.
Thus, the function $\Phi_{x,\psi}$ fulfils for every $\psi\in L^p(X)$ for almost all $x\in X$ that
\[ \Phi_{x,\psi}(w) = \int_X\tilde f(x,y,w,\psi(y))\d y+\int_Xg(y,x,\psi(y))\d y+\int_Xh(x,y)\d y \]
for all $w\in\R$ and is therefore convex because of the separate convexity of the function $\tilde f_{(x,y)}$ for almost all $(x,y)\in X\times X$. 

\medskip
Let us on the other hand assume that the function $\Phi_{x,\psi}$ is for every $\psi\in L^p(X)$ for almost all $x\in X$ convex.
Since the function $f_{(x,y)}$ is for almost all $(x,y)\in X\times X$ two times differentiable and we have the bounds~\eqref{eqBoundForFirstDerivative} and~\eqref{eqBoundForSecondDerivative} for its partial derivatives, we know that $\Phi_{x,\psi}$ is for every $\psi\in L^\infty(X)$ for almost all $x\in X$ two times differentiable and that we can differentiate under the integral sign. The convexity of $\Phi_{x,\psi}$ therefore implies for every $\psi\in L^\infty(X)$ that
\begin{equation}\label{eqPositivityOfSecondDerivative}
\frac{\d^2\Phi_{x,\psi}}{\d w^2}(w)=\int_X\partial_w^2f(x,y,w,\psi(y))\d y\ge0
\end{equation}
for almost all $x\in X$ and all $w\in\R$.

We now define for every $M\in(0,\infty)$ the measurable function
\[ \gamma_M:X\times X\times\R\to\R,\quad\gamma_M(x,y,w)=\min_{z\in[-M,M]}\partial_w^2f(x,y,w,z). \]
Then condition~\eqref{eqPositivityOfSecondDerivative} implies that
\begin{equation}\label{eqPositivityOfTheMeanOfGamma}
\int_X\gamma_M(x,y,w)\d y\ge0
\end{equation}
for almost all $x\in X$, all $w\in\R$, and all $M\in(0,\infty)$. 
To prove this, we first pick for arbitrary $M\in(0,\infty)$ and $\varepsilon\in(0,\infty)$ a measurable set $D\subset X$ (with arbitrarily small measure) and a $\lambda\in(0,\infty)$ such that the function $\alpha_M$ defined by~\eqref{eqBoundForSecondDerivative} fulfils
\[ \int_A\alpha_M(x,y)\d y < \frac\varepsilon3 \]
for every measurable set $A\subset X$ with $\L m(A)<\lambda$ and every $x\in X\setminus D$. 
By Scorza--Dragoni's theorem, see e.g.\ Theorem~$6.35$ in~\cite{FonsecaLeoni2007}, we can further choose a compact set $E\subset X\times X$ with $\L{2m}(X\times X\setminus E)<\frac12\lambda^2$ such that the restricted functions $\gamma_M|_{E\,\times\,[-M,M]}$ and $\partial_w^2f|_{E\,\times\,[-M,M]\,\times\,[-M,M]}$ are continuous. In particular, we find a constant $\delta\in(0,\infty)$ such that
\[ |\partial_w^2f(x,y,w,z)-\partial_w^2f(\tilde x,y,w,z)|<\frac\varepsilon{3\L m(X)} \]
for all $x,\tilde x,y\in X$ and $w,z\in[-M,M]$ with $(x,y)\in E$, $(\tilde x,y)\in E$, and $|x-\tilde x|<\delta$. Moreover, we define the set 
\[ F=\{x\in X\mid \L m(X\setminus E_x)\ge\tfrac\lambda2\}, \]
where $E_x=\{y\in X\mid (x,y)\in E\}$, $x\in X$, and remark that we have the estimate $\L m(F)<\frac2\lambda\L{2m}(X\times X\setminus E)<\lambda$.
By Aumann's measurable selection theorem, see e.g.\ Theorem~$6.10$ in~\cite{FonsecaLeoni2007}, we finally find for every $x\in X$ and every $w\in[-M,M]$ a measurable function $\psi_{x,w}:X\to[-M,M]$ with $\partial_w^2f(x,y,w,\psi_{x,w}(y))=\gamma_M(x,y,w)$ for almost all $y\in X$. 

Putting all this together, we get for every $x\in X\setminus(D\cup F)$ and every $\tilde x\in X\setminus F$ with $|x-\tilde x|<\delta$ that
\begin{align*}
\int_X\gamma_M(x,y,w)\d y&\ge\int_{E_x\cap E_{\tilde x}}\partial_w^2f(x,y,w,\psi_{x,w}(y))\d y-\int_{X\setminus E_x\cup X\setminus E_{\tilde x}}\alpha_M(x,y)\d y \\
&\ge\int_{E_x\cap E_{\tilde x}}\left(\partial_w^2f(\tilde x,y,w,\psi_{x,w}(y))-\frac\varepsilon{3\L m(X)}\right)\d y-\frac\varepsilon3.
\end{align*}
By Lebesgue's density theorem, we have for almost every point $x\in X\setminus(D\cup F)$ that the set of points $\tilde x\in X\setminus F$ with $|x-\tilde x|<\delta$ has positive measure. Because of condition~\eqref{eqPositivityOfSecondDerivative}, we therefore find for almost every $x\in X\setminus(D\cup F)$ and every $w\in[-M,M]$ a point $\tilde x\in X\setminus F$ with distance $|x-\tilde x|<\delta$ such that $\int_X\partial_w^2f(\tilde x,y,w,\psi_{x,w}(y))\d y\ge0$ and thus
\[ \int_X\gamma_M(x,y,w)\d y \ge-\varepsilon. \]
Letting now $\lambda$, $\varepsilon$, and the measure of the set $D$ tend to zero, we get~\eqref{eqPositivityOfTheMeanOfGamma}.

\medskip
Since the map $(0,\infty)\to\R$, $M\mapsto\gamma_M(x,y,w)$ is for all $x,y\in X$ and $w\in\R$ monotonically decreasing, we may define $\gamma(x,y,w)=\lim_{M\to\infty}\gamma_M(x,y,w)$. From the Lebesgue monotone convergence theorem, we further get that
\[ \int_X\gamma(x,y,w)\d y = \lim_{M\to\infty}\int_X\gamma_M(x,y,w)\d y\ge0. \]
With condition~\eqref{eqBoundForSecondDerivative}, this in particular implies for every $M\in(0,\infty)$ that
\[ \int_X|\gamma(x,y,w)|\d y \le 2\int_X\gamma^+(x,y,w)\d y \le 2\int_X|\partial_w^2f(x,y,w,0)|\d y\le2\int_X\alpha_M(x,y)\d y \]
for almost all $x\in X$ and all $w\in\R$ with $|w|\le M$, where $\gamma^+$ denotes the positive part of the function $\gamma$.
Therefore, defining the function $g:X\times X\times\R\to\R$ by
\[ g(x,y,w) = \int_0^w\int_0^{\tilde w}\left(\gamma(x,y,\hat w)-\frac1{\L m(X)}\int_X\gamma(x,\tilde y,\hat w)\d\tilde y\right)\d\hat w\d\tilde w, \]
Fubini's theorem implies for almost all $x\in X$ and all $w\in\R$ that $y\mapsto g(x,y,w)$ is integrable and fulfils
\begin{equation}\label{eqGIntegratesToZero}
\int_Xg(x,y,w)\d y = 0.
\end{equation}
Moreover, we have by construction
\[ \partial_w^2g(x,y,w)\le\gamma(x,y,w)\le\partial_w^2f(x,y,w,z) \]
for almost all $x,y\in X$, almost all $w\in\R$, and all $z\in\R$. Therefore, the function $\tilde f:X\times X\times\R\times\R\to\R$ defined by
\[ \tilde f(x,y,w,z)=f(x,y,w,z)-g(x,y,w)-g(y,x,z) \]
fulfils that $\tilde f_{(x,y)}$ is for almost all $(x,y)\in X\times X$ separately convex.

Thus, it only remains to prove that the function $g$ satisfies a condition of the form~\eqref{eqUpperBoundForG}. We find with the property~\eqref{eqGIntegratesToZero} of~$g$ that
\[ \int_X|g(x,y,w)|\d y=2\int_Xg^+(x,y,w)\d y\le2\int_X\left|\int_0^w\int_0^{\tilde w}\partial_w^2f(x,y,\hat w,0)\d\hat w\d\tilde w\right|\d y \]
for almost all $x\in X$ and all $w\in\R$. Using now the $p$-regularity of $f$, we find for every $M\in(0,\infty)$ positive functions $\tilde\alpha_M\in L^{p^*}(X)\otimes L^1(X)$ and $\tilde\beta\in L^1(X)$ such that
\begin{align*}
\int_X|g(x,y,w)|\d y&\le2\int_X\left|\int_0^w|\partial_wf(x,y,\tilde w,0)-\partial_wf(x,y,0,0)|\d\tilde w\right|\d y\\
&\le \int_X\left|\int_0^w(\tilde\alpha_M(x,y)+\tilde\beta(y)\pp{p-1}M(\tilde w))\d\tilde w\right|\d y
\end{align*}
for almost all $x\in X$ and all $w\in\R$.
If $p=\infty$, we then immediately find
\[ \int_X|g(x,y,w)|\d y \le M\int_X\tilde\alpha_M(x,y)\d y+\pp\infty M(w) \]
for almost all $x\in X$ and all $w\in\R$, and if $p\in[1,\infty)$, we apply Youngs inequality to get 
\[ \int_X|g(x,y,w)|\d y \le \frac1{p^*}\left(\int_X\tilde\alpha_M(x,y)\d y\right)^{p^*}+\frac1p\left(1+\int_X\tilde\beta(y)\d y\right)\p(w) \]
for almost all $x\in X$ and all $w\in\R$.
\end{proof}

However, for functionals of vector-valued functions, this argumentation fails. We give a counterexample to illustrate the problematic. For simplicity, we waive the symmetry of the function $f$ and consider only the case $n=2$.

\begin{example}
Let $X=[-1,1]$, $n=2$, and $p\ge2$. We choose a non-negative, convex function $a\in C^2(\R)$ with $a(\zeta)=|\zeta|-1$ for all $\zeta\in\R\setminus(-2,2)$ and define the function $b\in C^2(\R)$ by
\[ b(\zeta) = \begin{cases}1+\zeta+\frac12\zeta^2&\text{if}\;\zeta\ge0,\\(1-\zeta+\frac12\zeta^2)^{-1}&\text{if}\;\zeta<0.\end{cases} \]
Moreover, we define the function $f:X\times X\times\R^2\times\R^2\to\R$ by
\[ f(x,y,w,z) = \begin{cases}\frac12(b(z_1)w_1^2+b(-z_1)w_2^2)&\text{if}\quad y\ge0,\\\frac12a(z_1)(w_1^2+w_2^2)+z_1w_1w_2&\text{if}\quad y<0.\end{cases} \]
Then we find for the Hessian matrix $\mathrm H\Phi_{x,\psi}$ of the function $\Phi_{x,\psi}$ defined in~\eqref{eqDefinitionOfPhi} for every $x\in X$ and $\psi\in L^p(X;\R^2)$ the expression
\begin{align*}
\mathrm H\Phi_{x,\psi}(w)&=\int_{-1}^1\mathrm H_wf(x,y,w,\psi(y))\d y \\
&=\int_{-1}^0\begin{pmatrix}a(\psi_1(y))&\psi_1(y)\\\psi_1(y)&a(\psi_1(y))\end{pmatrix}\d y+\int_0^1\begin{pmatrix}b(\psi_1(y))&0\\0&b(-\psi_1(y))\end{pmatrix}\d y.
\end{align*}
Since $a(\zeta)\ge|\zeta|-1$ for all $\zeta\in\R$, we have for all $\zeta,\tilde\zeta\in\R$ that
\begin{align*}
\det\begin{pmatrix}a(\zeta)+b(\tilde\zeta)&\zeta\\\zeta&a(\zeta)+b(-\tilde\zeta)\end{pmatrix}&=a(\zeta)^2+(b(\tilde\zeta)+b(-\tilde\zeta))a(\zeta)+1-\zeta^2 \\
&\ge(a(\zeta)+1)^2-\zeta^2\ge0.\phantom\int
\end{align*}
Therefore, we find for every $x\in X$ and every $\psi\in L^p(X;\R^2)$ that 
\[ \det\mathrm H\Phi_{x,\psi}(w)\ge 0\quad\text{for all}\quad w\in\R^2 \]
and thus that the function $\Phi_{x,\psi}$ is convex.

We now want to show that there does not exist a measurable function $g:X\times X\times\R^2\to\R$ such that $g_{(x,y)}:\R^2\to\R$, $w\mapsto g(x,y,w)$ is for almost all $(x,y)\in X\times X$ twice continuously differentiable, 
\begin{equation}\label{eqConditionForg}
\int_Xg(x,y,w)\d y=0\quad\text{for almost all}\quad x\in X\quad\text{and all}\quad w\in\R^2,
\end{equation}
and the map $\tilde f_{(x,y,z)}:\R^2\to\R$, $\tilde f_{(x,y,z)}(w)=f(x,y,w,z)+g(x,y,w)$ is for almost all $(x,y)\in X\times X$ and all $z\in\R^2$ convex. So, assume there exists such a function $g$. 
Then for almost all $(x,y)\in X\times[-1,0)$ and all $z\in\R^2$, the convexity of $\tilde f_{(x,y,z)}$ implies that
\[ \det\mathrm H\tilde f_{(x,y,z)}(w) = \det\begin{pmatrix}a(z_1)+\partial_{w_1}^2g_{(x,y)}(w)&z_1+\partial_{w_1}\partial_{w_2}g_{(x,y)}(w)\\z_1+\partial_{w_1}\partial_{w_2}g_{(x,y)}(w)&a(z_1)+\partial_{w_2}^2g_{(x,y)}(w)\end{pmatrix}\ge0 \]
for all $w\in\R^2$. Using that $a(\zeta)=|\zeta|-1$ for $|\zeta|\ge2$, we find that this is only possible if we have for almost all $(x,y)\in X\times[-1,0)$ and all $w\in\R^2$ that
\[ \partial_{w_1}^2g_{(x,y)}(w)+\partial_{w_2}^2g_{(x,y)}(w)\ge2. \]
Because of the condition~\eqref{eqConditionForg}, this implies that there exists for every $w\in\R^2$ an $i\in\{1,2\}$ and a set $A\subset X\times[0,1]$ with positive measure such that $\partial_{w_i}^2g_{(x,y)}(w)\le-1$ for all $(x,y)\in A$. But then for every $(x,y)\in A$ and every $z\in\R^2$ with $(-1)^iz_1>0$, the Hessian matrix
\[  \mathrm H\tilde f_{(x,y,z)}(w) = \begin{pmatrix}b(z_1)+\partial_{w_1}^2g_{(x,y)}(w)&\partial_{w_1}\partial_{w_2}g_{(x,y)}(w)\\\partial_{w_1}\partial_{w_2}g_{(x,y)}(w)&b(-z_1)+\partial_{w_2}^2g_{(x,y)}(w)\end{pmatrix} \]
is not positive semidefinite.
\end{example}

\appendix

\section{Some Technicalities}
We give here two missing technicial details to the proofs in the previous sections.

We begin with the statement that we can cover almost one forth of every set $E\subset X\times X$ by a set of the form $A\times A\c$ with some measurable set $A\subset X$. To be more precise, let us introduce for every $N\in\N$ the notation
\[ Q_a^N(x)=\prod_{j=1}^N(x_j-\tfrac a2,x_j+\tfrac a2) \]
for the cube in $\R^N$ with side length $a\in(0,\infty)$ and center $x\in\R^N$. In the space $\R^m$, we further define for every $\delta\in(0,\infty)$ the checkerboard pattern $S_\delta\subset\R^m$ by
\begin{equation}\label{eqCheckerboardTiling}
S_\delta = \bigcup_{\{x\in\Z^m\mid\sum_{j=1}^mx_j\in2\Z\}}Q_\delta^m(x\delta).
\end{equation}

\begin{lemma}\label{thCheckerboardApproximation}
Let $E\subset X\times X$ be a measurable set. Then there exists for every $\varepsilon\in(0,\infty)$ a $\delta_0\in(0,\infty)$ such that
\[ \L{2m}((S_\delta\times S_\delta\c)\cap E)\ge\left(\frac14-\varepsilon\right)\L{2m}(E)\quad\text{for all}\quad\delta\in(0,\delta_0), \]
where $S_\delta\c=\R^m\setminus S_\delta$.
\end{lemma}
\begin{proof}
Let $\varepsilon\in(0,\infty)$ be arbitrarily given.
Since $E$ is a measurable set, we can cover it with pairwise disjoint cubes $Q_{a_i}^{2m}(\xi_i)$, $a_i\in\Q\cap(0,\infty)$, $\xi_i\in\Q^{2m}$, $i\in\N$, such that
\[ \L{2m}({\textstyle\bigsqcup_{i\in\N}}Q_{a_i}^{2m}(\xi_i)\setminus E)\le\frac\varepsilon4\L{2m}(E). \]
We further choose $\gamma\in(0,\infty)$ such that
\[ \L{2m}({\textstyle\bigsqcup_{\{i\in\N\mid a_i\le\gamma\}}}Q_{a_i}^{2m}(\xi_i))\le\varepsilon\L{2m}(E). \]

Since the set $S_\delta\times S_\delta\c$ covers for every $\delta\in(0,\infty)$ exactly one forth of every cube whose side length is an integer multiple of $2\delta$, we have for every $\delta\in(0,\infty)$, $a\in(2\delta,\infty)$, and $\xi\in\R^{2m}$ that
\[ \L{2m}((S_\delta\times S_\delta\c)\cap Q_a^{2m}(\xi))\ge\frac{(a-2\delta)^{2m}}4=\frac{(1-\frac{2\delta}a)^{2m}}4\L{2m}(Q_a^{2m}(\xi)). \]
So, with $\delta_0=(1-(1-2\varepsilon)^{\frac1{2m}})\frac\gamma2$, we get for every $\delta\in(0,\delta_0)$ that
\[ \L{2m}((S_\delta\times S_\delta\c)\cap Q_a^{2m}(\xi))\ge\left(\frac14-\frac\varepsilon2\right)\L{2m}(Q_a^{2m}(\xi)) \]
for every cube $Q_a^{2m}(\xi)$ with side length $a\in(\gamma,\infty)$ and arbitrary center $\xi\in\R^{2m}$, and therefore,
\begin{align*}
\L{2m}((S_\delta\times S_\delta\c)\cap E) &\ge \L{2m}((S_\delta\times S_\delta\c)\cap{\textstyle\bigsqcup_{\{i\in\N\mid a_i>\gamma\}}}Q_{a_i}^{2m}(\xi_i))-\frac\varepsilon4\L{2m}(E) \\
&\ge\left(\frac14-\frac\varepsilon2\right)\L{2m}({\textstyle\bigsqcup_{\{i\in\N\mid a_i>\gamma\}}}Q_{a_i}^{2m}(\xi_i))-\frac\varepsilon4\L{2m}(E) \\
&\ge\left(\frac14-\frac\varepsilon2\right)(\L{2m}(E)-\varepsilon\L{2m}(E))-\frac\varepsilon4\L{2m}(E) \\
&\ge\left(\frac14-\varepsilon\right)\L{2m}(E),
\end{align*}
as desired.
\end{proof}

In particular, this result shows that we can also choose for finitely many measurable sets $E_j\subset X\times X$, $j\in[1,N]\cap\N$, $N\in\N$, and arbitrary $\varepsilon\in(0,\infty)$ a $\delta\in(0,\infty)$ such that the set $S_\delta\times S_\delta\c$ fulfils 
\[ \L{2m}((S_\delta\times S_\delta\c)\cap E_j)\ge\left(\frac14-\varepsilon\right)\L{2m}(E_j) \]
for all $j\in[1,N]\cap\N$.

\medskip

The second lemma slightly generalises the result that if a measurable function $g:X\times\R^n\to\R$ fulfils an integral inequality of the form $\int_Ag(x,\omega(x))\d x\ge0$ for all sets $A\subset X$ and all $\omega\in L^\infty(X;\R^n)$, then $g(x,w)\ge0$ for almost every $x\in X$ and for all $w\in\R^n$.

\begin{lemma}\label{thFromIntegralToPointwise}
Let $g:X\times\R^n\to\R$ be a function such that the map $\R^n\to\R$, $w\mapsto g(x,w)$ is continuous for almost all $x\in X$, and such that there exists for every $M\in(0,\infty)$ a function $\alpha_M\in L^1(X)$ with $|g(x,w)|\le\alpha_M(x)$ for almost all $x\in X$ and all $w\in[-M,M]$.

If there exists for every subset $E\subset X$ with positive measure and every $M\in(0,\infty)$ a measurable subset $E'_M\subset E$ with positive measure such that
\[ \int_Ag(x,\omega(x))\d x\ge0 \]
for all measurable sets $A\subset E'_M$ and all $\omega\in L^\infty(X;\R^n)$ with $\|\omega\|_\infty\le M$, then
\[ g(x,w)\ge0 \]
for almost all $x\in X$ and all $w\in\R^n$.
\end{lemma}
\begin{proof}
We define for every $k\in\N$ and $M\in(0,\infty)$ the measurable set
\[ E_{k,M} =\bigcup_{\{w\in\Q^n\mid|w|\le M\}}\{x\in X\mid g(x,w)\le-\tfrac1k\}. \]
Let us assume by contradiction that there exists a set $E\subset X$ with positive measure such that we find for every $x\in E$ a value $w\in\R^n$ with $g(x,w)<0$. Then the union $\bigcup_{k,M\in\N}E_{k,M}\supset E$ has positive measure, too. We thus find some $k,M\in\N$ with $\L m(E_{k,M})>0$.

Now, by assumption, there exists a set $E_{k,M}'\subset E_{k,M}$ with positive measure such that
\begin{equation}\label{eqIntegralInequality}
\int_Ag(x,\omega(x))\d x\ge0
\end{equation}
for all measurable sets $A\subset E_{k,M}'$ and all $\omega\in L^\infty(X;\R^n)$ with $\|\omega\|_\infty\le M$.

On the other hand, using Aumann's measurable selection theorem, we can find a function $\omega\in L^\infty(E_{k,M}';\R^n)$ such that $\|\omega\|_\infty\le M$ and $g(x,\omega(x))\le-\tfrac1{2k}$ for almost all $x\in E_{k,M}'$, which clearly contradicts~\eqref{eqIntegralInequality}.
\end{proof}

\bibliography{citations}
\bibliographystyle{plain}
\end{document}